\documentclass[12pt]{extarticle}
\usepackage{setspace}
\usepackage[utf8]{inputenc}
\usepackage{geometry}
\usepackage{cite}
\usepackage{amsmath}
\usepackage{amssymb}
\usepackage{mathabx}
\usepackage{amsthm}
\usepackage{mathabx}
\usepackage[affil-it]{authblk}
\usepackage{bm}
\usepackage{scalerel,stackengine}
\stackMath
\usepackage{hyperref}
\newtheorem{Theorem}{Theorem}[section]
\newtheorem{Corollary}[Theorem]{Corollary}
\newtheorem{lemma}[Theorem]{Lemma}

\newtheorem{proposition}[Theorem]{Proposition}

\newtheorem{definition}[Theorem]{Definition}

\title{An Algebraic Brascamp--Lieb Inequality}
\author{Jennifer Duncan}
\affil{School of Mathematics, University of Birmingham}
\date{December 15, 2020}

\newcommand{\N}{\mathbb{N}}

\newcommand{\R}{\mathbb{R}}

\newcommand\reallywidehat[1]{%
\savestack{\tmpbox}{\stretchto{%
  \scaleto{%
    \scalerel*[\widthof{\ensuremath{#1}}]{\kern-.6pt\bigwedge\kern-.6pt}%
    {\rule[-\textheight/2]{1ex}{\textheight}}%WIDTH-LIMITED BIG WEDGE
  }{\textheight}% 
}{0.5ex}}%
\stackon[1pt]{#1}{\tmpbox}%
}
\newcommand{\textsbf}[1]{\textnormal{\textbf{#1}}}
\newcommand{\BL}{\textnormal{BL}}
\newcommand{\dist}{\textnormal{dist}}

\renewcommand{\epsilon}{\varepsilon}
\providecommand{\keywords}[1]
{
  \small	
  \textbf{\textit{Keywords---}} #1
}

\begin{document}

\maketitle
\begin{abstract}
   The Brascamp-Lieb inequalities are a very general class of multilinear inequalities, well-known examples of which being H\"older's inequality, Young's convolution inequality, and the Loomis-Whitney inequality. Conventionally, a Brascamp-Lieb inequality is defined as a multilinear Lebesgue bound on the product of the pullbacks of a collection of functions $f_j\in L^{q_j}(\R^{n_j})$, for $j=1,...,m$, under some corresponding linear maps $B_j$. This regime is now fairly well understood \cite{bcct}, and moving forward there has been interest in nonlinear generalisations, where $B_j$ is now taken to belong to some suitable class of nonlinear maps. While there has been great recent progress on the question of local nonlinear Brascamp-Lieb inequalities \cite{bennett2020nonlinear}, there has been relatively little regarding global results; this paper represents some progress along this line of enquiry. We prove a global nonlinear Brascamp--Lieb inequality for `quasialgebraic' maps, a class that encompasses polynomial and rational maps, as a consequence of the multilinear Kakeya--type inequalities of Zhang and Zorin-Kranich. We incorporate a natural affine-invariant weight that both compensates for local degeneracies and yields a constant with minimal dependence on the underlying maps. We then show that this inequality generalises Young's convolution inequality on algebraic groups with suboptimal constant.
\end{abstract}
\vspace{0.5cm}\\
\textbf{Declarations}\\
 This paper will form part of the author's PhD thesis under the supervision of Jonathan Bennett, whose guidance, patience, and support was invaluable to the production of this work, and was funded by a scholarship from the UK Engineering and Physical Sciences Research Council (EPSRC). The author would also like to thank the anonymous referee for their thorough and helpful feedback, Alessio Martini for offering many insightful suggestions that greatly improved the quality of the paper, and Karoline van Gemst for some stimulating discussions on related topics. 
\newpage
\section{Introduction}\label{intro}

A common feature of many problems studied in modern harmonic analysis is the presence of some underlying geometric object, examples including Kakeya inequalities, Fourier restriction theory, and generalised Radon transforms. Usually, this object is equipped with a measure that does not detect geometric features such as curvature or transversality, properties that are often highly relevant in the contexts we are considering. It has many times been found that incorporating a weight that tracks these geometric features in a suitable manner yields inequalities that require few geometric hypotheses and exhibit additional uniformity properties (in the context of generalised Radon-transforms and convolution with measures supported on submanifolds, see for example \cite{oberlin1999convolution, dendrinos2009universal, stovall2010endpoint, dendrinos2019p, gressman2019oberlin, christ2020endpoint,gressman2020p}, or in the context of Fourier restriction \cite{drury1985Fourier, carbery2002restriction, oberlin2004uniform, carbery2007restriction, bennett2017multilinear, bak2009restriction, dendrinos2010affine, hickman2014affine}). In particular, one often finds that if the geometric object in question may be parametrised by polynomials or rational functions, then the associated bounds will usually only depend on their degree, as observed in \cite{christ2020endpoint, dendrinos2009universal, dendrinos2010affine, dendrinos2019p, stovall2010endpoint} for example. \par
Our main theorem is another instance of this phenomenon, and is set in the context of a global nonlinear Brascamp--Lieb inequality, a term that we shall define in the next section. The underlying object in question is a collection of `quasialgebraic' maps, which is a class encompassing polynomial, rational, and algebraic maps. Like polynomials, a quasialgebraic map has an associated degree, and the bounds for the corresponding nonlinear Brascamp--Lieb inequalities we obtain depend only on these degrees, the underlying dimensions, and exponents.\par

\subsection{Brascamp-Lieb Inequalities: Linear and Nonlinear}

We shall begin with the definition of a linear Brascamp-Lieb inequality.

\begin{definition}
    Let $m,n,n_1,...,n_m\in\N$ and $p_j\in[0,1]$. Let $V$ be an $n$-dimensional Hilbert space, and, for each $1\leq j\leq m$, let $V_j$ be an $n_j$-dimensional Hilbert space. For each $1\leq j\leq m$, let $L_j:V\rightarrow V_j$ be a linear surjection. Define the $m$-tuples $\textsbf{L}:=(L_j)_{j=1}^m$ and $\textsbf{p}:=(p_j)_{j=1}^m$. We refer to the pair $(\textsbf{L},\textsbf{p})$ as a Brascamp--Lieb datum.\\
Given a Brascamp--Lieb datum $(\textsbf{L},\textsbf{p})$, we define the associated Brascamp--Lieb inequality as 
\begin{align}
    \int_V\prod_{j=1}^mf_j\circ L_j(x)^{p_j}d\lambda_V(x)\leq C\prod_{j=1}^m\left(\int_{V_j}f_j(x_j)d\lambda_{V_j}(x_j)\right)^{p_j},\label{linearBL}
\end{align}
where $\lambda_X$ denotes the induced Lebesgue measure associated to a Hilbert space $X$. We let $\BL(\textsbf{L},\textsbf{p})$ denote the smallest constant $C$ such that (\ref{linearBL}) holds for all $f_j\in L^1(V_j)$.
\end{definition}

These arise as a natural generalisation of many familiar multilinear inequalities from mathematical analysis, such as H\"older's inequality, Young's convolution inequality, and the Loomis-Whitney inequality, and became of greater interest within harmonic analysis once the role that transversality plays in multilinear restriction theory became apparent (see the survey article \cite{bennett2013aspects} for further discussion). Their study was initiated in the '70s by Brascamp, Lieb, and Luttinger \cite{brascamp1974general}, later continued by the authors of, for example, \cite{ball1989volumes, barthe1998optimal, carlen2004sharp}. Since then, necessary and sufficient conditions for finiteness and extremisability of Brascamp-Lieb inequalities were established by Bennett, Carbery, Christ, and Tao in \cite{bcct}. Additionally, some far-reaching connections with Brascamp-Lieb inequalities have been found in convex geometry \cite{conv}, kinetic theory \cite{carlen2004sharp}, number theory \cite{guo2018integer}, computer science \cite{garg2018algorithmic}, and group theory \cite{christ2015holder}.\par
It is common in applications to encounter nonlinear variants where the linear maps $L_j$ are replaced with nonlinear maps, as observed in \cite{carlen2004sharp, bennett2005non, bennett2010some, bennett2018stability} for example. Significant progress on inequalities of this type was made in \cite{bennett2020nonlinear}, where the authors establish the following highly general local nonlinear Brascamp-Lieb inequality.
 \begin{Theorem}[Bennett, Bez, Buschenhenke, Cowling, Flock 2018]
        For each $j\in\{1,...,m\}$, let $B_j:U\rightarrow M_j$ be a $C^2$ submersion defined on an open neighbourhood of a point $x_0\in\R^n$. For each $\epsilon>0$, there exists a $\delta>0$ such that 
        $$\int_{|x-x_0|\leq\delta}\prod_{j=1}^mf_j\circ B_j(x)^{p_j}dx\leq(1+\varepsilon)\BL(\textsbf{dB}(x_0),\textsbf{p})\prod_{j=1}^m\left(\int_{\R^{n_j}}f_j\right)^{p_j}.$$
 \end{Theorem}

The study of global Brascamp-Lieb inequalities is currently at the stage of case by case examples, which include an inequality proved by Bennett, Bez, and Gutierrez for nonlinear data of degree one \cite{bennett2013global} and a certain global trilinear inequality of Koch and Steinerberger \cite{koch2015convolution}. It was first suggested in \cite{bennett2005non} that a global Brascamp-Lieb inequality should include an appropriate weight factor in order to compensate for local degeneracies, and it is upon this suggestion that we include a weight factor of the form $\BL(\textsbf{dB}(x),\textsbf{p})^{-1}$ in our inequality. It was also discussed in the same paper that even with an appropriate weight factor one cannot expect a global nonlinear Brascamp-Lieb inequality to hold with only local hypotheses, due to reasons relating to infinite failure of injectivity. We address this issue by imposing that our nonlinear maps are \textit{quasialgebraic}, a property we define in the following section, that entails that the fibres of our maps are algebraic varieties, the heuristic motivation being that B\'ezout's theorem then eliminates such global injectivity issues.

\subsection{Preliminary Definitions and Notation}
Before we define the notion of a quasialgebraic map, we should first clarify the notion of an algebraic variety.
\begin{definition}
A subset $H\subset \R^n$ is an algebraic variety in $\R^n$ if and only if there exists a finite collection of polynomials $\mathcal{P}\subset\R[x_1,...,x_n]$ such that
\begin{align}
   H=\{x\in \R^n: p(x)=0 \hspace{5pt}\forall p\in\mathcal{P}\} \label{eq:varietydef}
\end{align}
We then define the \textnormal{degree} of $H$ to be the minimum of the quantity $\max_{p\in\mathcal{P}}\deg p$ as $\mathcal{P}$ ranges over all collections of polynomials such that (\ref{eq:varietydef}) holds.
\end{definition}
For instance, any finite set of points is an algebraic variety, and its degree is equal to its cardinality. By the implicit function theorem, if $M$ admits a defining vector-valued polynomial $\textsbf{p}=(p_1,...,p_{n-d}):\R^n\rightarrow\R^{n-d}$ whose derivative has full rank at a point $p\in M$, then $M$ is locally a $d$-dimensional manifold near $p$, and we refer to such $p$ as non-singular points of $M$. If the non-singular points of $M$ form an open and dense subset of $M$, we shall refer to $M$ as a $d$-dimensional variety. We remark that, while being perfectly suitable for our purposes, this is a restricted definition, and would be more widely referred to as the definition of a real affine variety. A more general definition of an algebraic variety can be found in \cite{hartshorne1977graduate} for example.
\begin{definition}\label{quasi}
   Let $M\subset\R^n$ be an open subset of a $d$-dimensional algebraic variety and let $N$ be a Riemannian manifold. We say that a map $F:M\rightarrow N$ that is $C^{\infty}$ on an open dense subset of $M$ is $\textnormal{quasialgebraic}$ if its fibres are open subsets of algebraic varieties. We define the degree of $F$ to be the maximum degree of its fibres (this may be infinite).
\end{definition}
The author is not aware of this notion of a quasialgebraic map being discussed anywhere in the literature, however this is not to pretend that it is an innovative concept, merely one that is very much tailored to our purposes. As remarked earlier, the class of quasialgebraic maps encompasses many important classes of maps,
as ordered below.
$$\{\textit{polynomial maps}\}\subset\{\textit{rational maps}\}\subset\{\textit{algebraic maps}\}\subset\{\textit{quasialgebraic maps}\}$$
As one would hope, the notion of degree in Definition \ref{quasi} coincides with the conventional notion of degree for each of the above classes. It is easy to check that, unlike the classes of polynomial, rational, and algebraic maps, the class of quasialgebraic maps is `closed' under diffeomorphism, in the sense that given a quasialgebraic map $F:M\rightarrow N$, and a diffeomorphism $\phi:N\rightarrow N'$, the map $F':=\phi\circ F:M\rightarrow N'$ is a quasialgebraic map of the same degree as $F$.\par\vspace{0.2cm}

 Before moving onto stating our main theorem we should first state our notational conventions. In this paper, the expression `$A\lesssim B$' will be used to denote that `$A\leq CB$', where $C>0$ is a constant depending only upon the relevant dimensions and exponents, and the expression `$A\simeq B$' will be used to denote that `$A\lesssim B\lesssim A$'. Given a metric space $M$, We let $U_r(x)$ denote an open ball of radius $r>0$ centred at a point $x\in M$, and we denote the centred dilate of a ball $V$ by a factor $c>0$ by $cV$. Notice that at some points either $dB_j$ will not be defined or will fail to be surjective; in such cases we set $\BL(\textsbf{dB}(x),\textsbf{p})=\infty$. Given a Brascamp-Lieb datum $(\textsbf{L},\textsbf{p})$ such that $L_j:V\rightarrow V_j$ and a subspace $W\leq V$, we let $\BL_W(\textsbf{L},\textsbf{p})$ denote the best constant $C>0$ in the following `restricted' Brascamp-Lieb inequality.

\begin{align} \int_W \prod_{j=1}^mf_j\circ L_j(x)^{p_j}d\lambda_W(x)\leq C\prod_{j=1}^m\left(\int_{L_jW}f_j(x_j)d\lambda_{L_jW}(x_j)\right)^{p_j}.\label{eq: restbraslieb}\end{align}

Lastly, we shall denote the zero-set of a polynomial map $p:\R^n\rightarrow\R^k$ by $Z(p):=\{x\in\R^n:p(x)=0\}$.

\subsection{Main Results}

We shall now state our main theorem.

\begin{Theorem}[Quasialgebraic Brascamp--Lieb Inequality]\label{algbras}\label{genalgbras}
Let $d,m,n\in\N$ and, for each $1\leq j\leq m$, let $n_j\in\N$ and $p_j\in[0,1]$. Assume that the scaling condition $\sum_{j=1}^mp_jn_j=d$ is satisfied. Let $M\subset\R^n$ be an open subset of a $d$-dimensional algebraic variety, and for each $j\in\{1,...,m\}$, let $ M_j$ be an $n_j$-dimensional Riemannian manifold.\par
We consider quasialgebraic maps $B_j:M\rightarrow  M_j$ that extend to quasialgebraic maps on some open set $A\subset\R^n$. Setting $\textsbf{p}:=(p_1,...,p_m)$ and equipping each $ M_j$ with the measure $\mu_j$ induced by its Riemannian metric, the following inequality holds for all $f_j\in L^1( M_j)$:
\begin{align}
    \int_M\prod_{j=1}^mf_j\circ B_j(x)^{p_j}\frac{d\sigma(x)}{\BL_{T_xM}(\textsbf{dB}(x),\textsbf{p})}\lesssim\deg(M)\prod_{j=1}^m\left(\deg(B_j)\int_{ M_j}f_j(x_j)d\mu_j(x_j)\right)^{p_j},\label{eq:algbras}
\end{align}
where $\sigma$ is the induced $d$-dimensional Hausdorff measure on $M$.
\end{Theorem}
Notice that we impose no local condition on the maps $B_j$, not even that they are submersions. This is allowed because the weight we have incorporated on the left-hand side vanishes when the maps $B_j$ degenerate, hence we do not have to worry about counterexamples where the functions $f_j$ concentrate on points whose fibres admit poorly behaved intersections. One should also observe the similarity between this weight and the weight used in Theorem 1 of \cite{gressman2020p}. In particular, this Theorem immediately gives us a less powerful, but more concisely stated weighted nonlinear Brascamp-Lieb inequality for polynomial maps.
\begin{Corollary}[Polynomial Brascamp-Lieb Inequality]
    Let the dimensions and exponents be as in Theorem \ref{genalgbras}, and let $B_j:\R^d\rightarrow\R^{n_j}$ be polynomial maps. The following inequality holds over all $f_j\in L^1(\R^{n_j})$:
    \begin{align}
    \int_{\R^d}\prod_{j=1}^mf_j\circ B_j(x)^{p_j}\frac{dx}{\BL(\textsbf{dB}(x),\textsbf{p})}\lesssim\prod_{j=1}^m\left(\deg(B_j)\int_{\R^{n_j}}f_j(x_j)dx_j\right)^{p_j}.
\end{align}
\end{Corollary}
Brascamp-Lieb inequalities were first studied as a generalisation of Young's convolution inequality on $\R^n$ in \cite{brascamp1974general}, it is therefore fitting that one may view Theorem \ref{genalgbras} as a generalisation of Young's convolution inequality on algebraic groups, those being algebraic varieties equipped with a group structure such that the associated multiplication and inversion maps are `morphisms' of varieties, i.e. restrictions of polynomial maps. 

\begin{Corollary}\label{youngs}
Let $G$ be an algebraic group, with left-invariant Haar measure $d\mu$. We let $\Delta:G\rightarrow(0,\infty)$ be the modular character associated to $(G,\mu)$, which is the unique homomorphism such that for all measurable $f:G\rightarrow\R$,
$$\int_Gf(x)d\mu(x)=\Delta(g)\int_Gf(xg)d\mu(x).$$
We define left-convolution as follows: 
$$f\ast g(x):=\int_G f(xy^{-1})g(y)d\mu(y)$$

The inequality \eqref{eq:youngs} holds for all $p_1,...,p_m,r\in[1,\infty]$ such that $\frac{1}{r'}=\sum_{j=1}^m\frac{1}{p'_j}$, and all $f_j\in L^{p_j}(G)$,
\begin{align}
    \left\Vert \Asterisk_{j=1}^mf_j\Delta^{\sum_{l=1}^{j-1}\frac{1}{p_l'}}\right\Vert_{L^r(G)}\lesssim\deg(G)\deg(m_G)^{\sigma}\prod_{j=1}^m\Vert f_j\Vert_{L^{p_j}(G)}\label{eq:youngs}
\end{align}
where $m_G:G\times G\rightarrow G$ is the multiplication operation, and $\sigma:=\sum_{j=1}^m\frac{1}{p_j}$.
\end{Corollary}
 We give a proof of this corollary in section \ref{appendix}. It is important to note that since the best constant for Young's inequality on locally compact topological groups is always less than 1 \cite{terp2017p}, Corollary \ref{youngs} does not offer any improvement to the theory, however it is nonetheless included in this paper for the sake of context; we refer the reader to \cite{quek1984sharpness, inoue1992lp, terp2017p, cowling2019hausdorff} for further details on Young's inequality in abstract settings. We remarked earlier on that Theorem \ref{genalgbras} is an example of an affine-invariant inequality, in the sense that the left-hand side is invariant under the natural action $A:B_j\mapsto B_j\circ A$ of $GL_n(\R)$ on the class of quasialgebraic data, however this inequality in fact exhibits a more general diffeomorphism-invariance property, as described by the following proposition.
\begin{proposition}\label{difinv}
Let the dimensions and exponents be as in Theorem \ref{algbras}. Let $M$ and $\widetilde{M}$ be $d$-dimensional Riemannian manifolds equipped with induced measures $\mu$ and $\widetilde{\mu}$, and, for each $1\leq j\leq m$, let $ M_j$ be an $n_j$-dimensional Riemannian manifold. Let $B_j:M\rightarrow  M_j$ be a.e. $C^1$, and $\phi:\widetilde{M}\rightarrow M$ be a diffeomorphism. Defining $\widetilde{\textsbf{B}}=(\tilde{B}_j)_{j=1}^m=(B_j\circ\phi)_{j=1}^m$, the following then holds for all $f_j\in L^1( M_j)$:
\begin{align*}
    \int_{\widetilde{M}}\prod_{j=1}^mf_j\circ \widetilde{B}_j(x)^{p_j}\frac{d\widetilde{\mu}(x)}{\BL_{T_x\widetilde{M}}(\textsbf{d}\widetilde{\textsbf{B}}(x),\textsbf{p})}=\int_{M}\prod_{j=1}^mf_j\circ B_j(x)^{p_j}\frac{d\mu(x)}{\BL_{T_x M}(\textsbf{dB}(x),\textsbf{p})}
\end{align*}
\end{proposition}
\begin{proof}
    By the chain rule and Lemma 3.3 of \cite{bcct}, for almost every $x\in M$, $$\BL_{T_x\widetilde{M}}(\textsbf{d}\widetilde{\textsbf{B}}(x),\textbf{p})=\BL_{T_{\phi(x)}M}(\textsbf{dB}(\phi(x))d\phi(x),\textsbf{p})=\BL_{T_{\phi(x)}M}(\textsbf{dB}(\phi(x)),\textsbf{p})\det(d\phi(x))^{-1}.$$ Hence, by changing variables we obtain that 
    \begin{align*}
        \int_{\widetilde{M}}\prod_{j=1}^mf_j\circ \widetilde{B}_j(x)^{p_j}\frac{dx}{\BL_{T_x\widetilde{M}}(\textsbf{d}\widetilde{\textsbf{B}}(x),\textsbf{p})}&=\int_{\widetilde{M}}\prod_{j=1}^mf_j\circ \widetilde{B}_j(x)^{p_j}\frac{\det(d\phi(x))dx}{\BL_{T_{\phi(x)}M}(\textsbf{d}\textsbf{B}(\phi(x)),\textsbf{p})}\\
        &=\int_M\prod_{j=1}^mf_j\circ B_j(x)^{p_j}\frac{dx}{\BL_{T_x M}(\textsbf{dB}(x),\textsbf{p})}.
    \end{align*}
\end{proof}

In light of Proposition \ref{difinv}, one may extend Theorem \ref{algbras} to any $m$-tuple of maps $(B_j)_{j=1}^m$ that may each be written as a composition of a quasialgebraic map with a common diffeomorphism $\phi$, however we shall leave this as a remark.
 The proof strategy for Theorem \ref{genalgbras} will be to appeal to a generalised endpoint multilinear curvilinear Kakeya inequality, which we will view as a discrete version of (\ref{eq:algbras}), and run a limiting argument in order to recover the full inequality.

\subsection{Endpoint Multilinear Kakeya-type Inequalities}\label{kakeya}

The tools we will be using in this proof trace their lineage back to the endpoint multilinear Kakeya inequality, conjectured by Bennett, Carbery, and Tao in \cite{bct}, later proved by Guth in \cite{guth2010endpoint}.
\begin{Theorem}[Endpoint Multilinear Kakeya Inequality, Guth (2009)]\label{guthkak}
For each $1\leq j\leq n$, let $\mathcal{T}_j$ be a collection of straight doubly infinite tubes $T_j\subset\R^n$ of unit width. Denote the direction of a tube $T_j\in\mathcal{T}_j$ by $e(T_j)$, and suppose that there exists $\theta>0$ such that, for any configuration of tubes $(T_1,...,T_n)\in\mathcal{T}_1\times...\times\mathcal{T}_n$, we have the uniform transversality bound $|\bigwedge_{j=1}^n e(T_j)|>\theta$, then the following inequality holds:
\begin{align}
    \int_{\R^n}\left(\prod_{j=1}^n\sum_{T_j\in\mathcal{T}_j}\chi_{T_j}\right)^{\frac{1}{n-1}}dx\lesssim \theta^{-\frac{1}{n-1}}\prod_{j=1}^n\left(\#\mathcal{T}\right)^{\frac{1}{n-1}}\label{eq:guthkak}
\end{align}
\end{Theorem}
Remarkably, the proof of this theorem relies heavily on sophisticated techniques from algebraic topology. 
If we suppose that each $T_j\in\mathcal{T}_j$ is parallel to the $j$-th axis, then we may interpret the tubes $T_j$ as preimages of balls $V_j\subset\R^{n-1}$ under the projection $\pi_j$ onto the orthogonal complement of the $j$-th coordinate axis, as such we may write $\sum_{T_j\in\mathcal{T}_j}\chi_{T_j}=\sum_{V_j\in\mathcal{V}_j}\chi_{V_j}\circ\pi_j$ for some collection $\mathcal{V}_j$ of unit balls $V_j$ in $\R^{n-1}$, from which we recover the Loomis-Whitney inequality via rescaling and applying a standard density argument.\par
Motivated by seeking a more simple proof of this theorem, Carbery and Valdimarsson later established the following affine-invariant generalisation via the Borsuk--Ulam theorem \cite{carbery2013endpoint}.
\begin{Theorem}[Affine-invariant Multilinear Kakeya, Carbery-Valdimarsson (2013) \cite{carbery2013endpoint}]\label{CarbVald}
Let $1\leq m\leq n$. For each $1\leq j\leq m$, let $\mathcal{T}_j$ be a collection of straight doubly infinite tubes $T_j$ of unit width. Then, the following inequality holds:
\begin{align}
    \int_{\R^n}\left(\sum_{(T_1,...,T_m)\in\mathcal{T}_1\times...\times\mathcal{T}_m}|\bigwedge_{j=1}^m e(T_j)|\chi_{T_1\cap...\cap T_m}\right)^{\frac{1}{m-1}}dx\lesssim\prod_{j=1}^m\left(\#\mathcal{T}\right)^{\frac{1}{m-1}}\label{eq:CarbVald}
\end{align}
\end{Theorem}
If we may uniformly bound the weight $|\bigwedge_{j=1}^m e(T_j)|$ below by some $\theta>0$, then this will allow us to factorise the integrand on the left-hand side of $(\ref{eq:CarbVald})$ in such a manner that we then recover Theorem \ref{guthkak}.
Zhang offers a generalisation of Theorem \ref{CarbVald} in \cite{zhang2017endpoint}, where, essentially, the tubes $T_j$ are replaced with tubular neighbourhoods of algebraic varieties. An early example of curvilinear variants of Kakeya inequalities of this kind is offered by Bourgain and Guth in \cite{bourgain2011bounds}, where they prove a trilinear inequality for algebraic curves ($1$-dimensional algebraic varieties) in $\R^4$.
\begin{Theorem}[Bourgain-Guth 2011 \cite{bourgain2011bounds}]\label{bourgainguth}
Suppose that $\Gamma_i\subset\R^4$ is an algebraic curve restricted to the unit 4-ball with degree $\lesssim 1$ and
$C^2$ norm $\lesssim 1$. Let $T_i$ denote the $\delta$-neighborhood of an algebraic curve $\Gamma_i$ and let $\mathcal{T}$ be an arbitrary finite set of such $T_i$. Define approximate tangent
vectors $v_i(x)$ for $x\in T_i\in\mathcal{T}$. The
following estimate holds:
\begin{align}
    \int_{U_1(0)}\left(\sum_{(T_i,T_j,T_k)\in\mathcal{T}^3}|v_i(x) \wedge v_j(x) \wedge v_k(x)|\chi_{T_i\cap T_j\cap T_k}(x)\right)^{\frac{1}{2}}dx\lesssim \delta^4(\#\mathcal{T})^{\frac{3}{2}} \label{eq: bourgainguth}
\end{align}

\end{Theorem}
There are higher-dimensional generalisations of this inequality due to Zhang and Zorin-Kranich, but before we state them, we remark that any higher-dimensional analogue of (\ref{eq: bourgainguth}) must involve some suitable generalisation of the wedge term in the integrand that tracks the transversality of the varieties in a similar manner. One such generalisation involves a weight that takes the form of a `wedge product' of the tangent spaces of the varieties, which we shall now define.
\begin{definition}
Let $W_1,...,W_m$ be a collection of subspaces of $\mathbb{R}^n$, and for each $W_j$ choose an orthonormal basis $w_1^j,..., w_{k_j}^j$. Observing that the $\sum_{j=1}^mk_j$-dimensional volume of the parallelepiped generated by the union of these bases, given by $|\bigwedge_{j=1}^m\bigwedge_{i=1}^{k_i}w_i^j|$, does not depend on the choice of bases, we denote this quantity by $|\bigwedge_{j=1}^mW_j|$.
\end{definition}
\begin{Theorem}[$k_j$-variety theorem, Zhang 2015 \cite{zhang2017endpoint}]\label{zhang}
Assume that $\sum_{j=1}^mk_j=n$. For each $j\in\{1,...,m\}$, let $H_j$ be an open subset of a $k_j$-dimensional algebraic subvariety in $\mathbb{R}^n$, and let $\sigma_j$ denote the $k_j$-dimensional Hausdorff measure on $H_j$, then,
\begin{align}
    \int_{\mathbb{R}^n}\left(\int_{H_1\cap U_1(x)\times...\times H_m\cap U_1(x)}|\bigwedge_{i=1}^mT_{y_j}H_j|d\sigma_1(y_1)...d\sigma_m(y_m)\right)^{\frac{1}{m-1}}dx\lesssim\prod_{j=1}^m\deg(H_j)^{\frac{1}{m-1}}\label{eq:zhang}
\end{align}
\end{Theorem}
While at first glance this inequality appears to have a very different form to (\ref{eq:CarbVald}) and (\ref{eq: bourgainguth}), one may view the inner integral as a weighted bump function supported in the intersection of the unit neighborhoods of the varieties $H_1,...,H_m$, where this weight is a higher-dimensional generalisation of the wedge of tangent vectors arising in (\ref{eq: bourgainguth}). We should remark that, in the same paper, Zhang does prove a stronger theorem than the above that accounts for more general configurations of dimensions and exponents, wherein the weight explicitly takes the form of a Brascamp--Lieb constant. Later, Zorin-Kranich devised a reformulation of this generalised theorem that makes use of Fremlin tensor product norms, and this is the version we shall be using to prove Theorem \ref{algbras}.
\begin{definition}
Given measure spaces $X_1,...,X_m$ and $p_j\in[1,\infty]$, define the Fremlin tensor product norm $\Vert F\Vert_{\widebar{\bigotimes}_{j=1}^mL^{p_j}(X_j)}$ on $\bigotimes_{j=1}^mL^{p_j}(X_j)$ by
$$\Vert F\Vert_{\widebar{\bigotimes}_{j=1}^mL^{p_j}(X_j)}:=\inf\left\lbrace\prod_{j=1}^m\Vert F_j\Vert_{L^{p_j}(X_j)}: F_j\in L^{p_j}(X_j), |F|\leq |F_1|\otimes...\otimes |F_m| \right\rbrace$$
\end{definition}
Zorin-Kranich also makes use of a non-standard regime for defining Brascamp--Lieb inequalities that takes, as data, collections of subspaces as opposed to linear maps, one that we shall now define. Given a collection of subspaces $W_1,...,W_m\leq\R^n$ such that $\dim(W_j)=k_j$, with a corresponding collection of exponents $p_1,...,p_m>0$, the associated `Brascamp--Lieb inequality' is defined as follows over all $f_j\in L^1(\R^n/W_j)$:
\begin{align}
    \int_{\R^n}\prod_{j=1}^mf_j(x+W_j)^{p_j}dx\leq C\prod_{j=1}^m\left(\int_{\R^n/W_j}f_j\right)^{p_j}
\end{align}
Following the notation of \cite{zorin2019Kakeya}, we then write $\overrightarrow{W_j}=(W_1,...,W_m)$, $\textsbf{p}:=(p_1,...,p_m)$, and denote the best constant $C>0$ in the above inequality by $BL'(\overrightarrow{W_j},\textsbf{p})$. In his paper, Zorin-Kranich makes use of local versions of the Brascamp-Lieb constants, which allows for exponents to lie outside of the polytope defined by the scaling condition $\sum_{j=1}^mp_jn_j=n$. We shall however state a version of Zorin-Kranich's theorem that assumes such a scaling condition, but nonetheless is more general than Theorem \ref{zhang}.
\begin{Theorem}[Zorin-Kranich 2017 \cite{zorin2019Kakeya}]\label{zorin}
Let $\mathcal{Q}$ be a decomposition of $\R^n$ into unit cubes and for each $1\leq j\leq m$, let $H_j\subset\R^n$ be an open subset of a $k_j$-dimensional algebraic variety and $p_j\in[0,1]$ be chosen such that $\sum_{j=1}^mp_j(n-k_j)=n$. Suppose that $P:=\sum_{j=1}^m p_j \geq 1$, then the following inequality holds:
\begin{align}
    \sum_{Q\in\mathcal{Q}}\Vert BL'(\overrightarrow{T_{x_j}H_j},\textsbf{p})^{-\frac{1}{P}}\Vert^P_{\widebar{\bigotimes}_{j=1}^mL^{P/p_j}_{x_j}(H_j\cap Q)}\lesssim\prod_{j=1}^m\deg(H_j)^{p_j}\label{eq:zorin0}
\end{align}
\end{Theorem}
Consequently, averaging over all axis-parallel choices of $\mathcal{Q}$ and rescaling by a factor of $2$ via the forthcoming Lemma \ref{rescale}, we obtain the following inequality under the same conditions:
\begin{align}
\int_{\R^n}\Vert BL'(\overrightarrow{T_{x_j}H_j},\textsbf{p})^{-\frac{1}{P}}\Vert^P_{\widebar{\bigotimes}_{j=1}^mL^{P/p_j}_{x_j}(H_j\cap U_1(x))}dx\lesssim\prod_{j=1}^m\deg(H_j)^{p_j}.\label{eq:zorin}
\end{align}
This integral representation is the form we shall be using in this paper.
In analogy with the discussion following the statement of Theorems \ref{guthkak} and \ref{CarbVald}, it is natural that one should attempt to derive a Brascamp--Lieb inequality from Theorems \ref{zhang} or \ref{zorin} by formally running the same argument as in the linear case. However, in the presence of nonlinearity, tubular neighbourhoods of fibres cannot be written as preimages of balls, hence we cannot immediately run the same density argument as before. We therefore need to use a more detailed construction, where we cover these preimages by a union of many very thin tubular neighbourhoods of fibres, paying careful attention to how they overlap (see figure \ref{fig: 3}, section \ref{heuristic}).\par

\section{Setup for the proof of Theorem \ref{genalgbras}}\label{setup}

\subsection{Reductions}
We shall assume for the remainder of the paper without loss of generality that the maps $B_j$ have finite degree, since the case of infinite degree holds vacuously, and that $\BL_{T_xM}(\textsbf{dB}(x),\textsbf{p})<\infty$ for all $x\in M$, in particular that $B_j$ is a submersion on $M$. We may do this firstly because we may remove the set of non-smooth points harmlessly since it is closed and null, so $M$ is still an open subset of an algebraic variety, and secondly we may remove the set of smooth points at which the weight arising in (\ref{genalgbras}) vanishes, i.e. those $x\in M$ such that $\BL(\textsbf{dB}(x),\textsbf{p})=\infty$, since this set is closed by continuity of the reciprocal of the Brascamp-Lieb constant (Theorem 5.2 of \cite{bennett2020nonlinear}).\par
We shall begin by reducing to the case where $d=n$, i.e. where  $M$ is an open subset of $\R^n$.
We begin with a standard geometric lemma.
\begin{lemma}\label{sigmapprox}
    Let $N$ be an $(n-d)$-dimensional Riemannian manifold and let $\chi_{\delta}:N\rightarrow\R$ be the normalised characteristic function associated to the $\delta$-ball centred at some fixed $z_0\in N$, defined by $\chi_{\delta}(z):=\delta^{n-d}\chi_{U_{\delta}(z_0)}$. Given an open set $A\subset\R^n$ and a submersion $B:A\rightarrow N$, then for any continuous and integrable $f:A\rightarrow\R$ the following holds:
    $$\int_{A}f(x)\chi_{\delta}\circ B(x)dx\overset{\delta\rightarrow 0}{\longrightarrow}\int_{A\cap B^{-1}(\{z_0\})}f(x)\det(dB(x)dB(x)^*)^{-\frac{1}{2}}d\sigma(x),$$
    where $d\sigma$ denotes the induced $d$-dimensional Hausdorff measure.
\end{lemma}

We also require the following identity of Brascamp-Lieb constants, which may be regarded as a crude example of a Brascamp-Lieb constant splitting through a critical subspace, a phenomenon that was studied in its full generality in \cite{bcct}.
\begin{lemma}\label{BLconstants}
    Let $d,n,m\in\N$, $n_1,...,n_m\in\N$ and write $n_{m+1}=n-d$. For $1\leq j\leq m+1$, we consider linear surjections $L_j:\R^n\rightarrow \R^{n_j}$ such that, for $1\leq j\leq m$, $L_j$ restricts to a surjection on the subspace $V:=\ker(L_{m+1})$. Let $p_j\in[0,1]$ for $1\leq j\leq m$ and $p_{m+1}=1$, and assume that the scaling condition $\sum_{j=1}^{m+1}p_jn_j=n$ is satisfied. Let $\widetilde{\textsbf{L}}:=(L_j)_{j=1}^{m+1}$ and $\widetilde{\textsbf{p}}:=(p_j)_{j=1}^{m+1}$. Then, the scaling condition $d=\sum_{j=1}^mp_jn_j$ holds. Furthermore, if we let $\textsbf{L}:=(L\vert_V)_{j=1}^m$ and $\textsbf{p}:=(p_j)_{j=1}^m$, we then have the following identity:
    $$\BL(\widetilde{\textsbf{L}},\widetilde{\textsbf{p}})=\det(L_{m+1}L_{m+1}^*)^{-\frac{1}{2}}\BL(\textsbf{L},\textsbf{p}).$$
\end{lemma}
The proofs of these lemmas are given in Section \ref{appendix}. Combining them with Theorem \ref{algbras} in the euclidean case then yields the general case.
\begin{proposition}
    If Theorem \ref{algbras} holds for $d=n$, then Theorem \ref{algbras} holds for general $d$.
\end{proposition}
\begin{proof}
    Let $B_{m+1}:\R^n\rightarrow\R^{n-d}$ be a polynomial map such that $M$ is an open subset of $Z(B_{m+1})$, and that $\deg(B_{m+1})=\deg(M)$. Let $A\subset\R^n$ be any bounded open set such that $B_{m+1}$ restricts to a submersion on $A\cap M$. Recall the definition of $\chi_{\delta}$ from Lemma \ref{sigmapprox}. By Lemmas \ref{sigmapprox} and \ref{BLconstants}, we know that given any $f_j\in C_0^{\infty}( M_j)$,
    \begin{align*}
         \int_{A\cap M}\prod_{j=1}^mf_j\circ B_j(x)^{p_j}\frac{d\sigma(x)}{\BL_{T_xM}(\textsbf{dB}(x),\textsbf{p})}&=\int_{A\cap M}\prod_{j=1}^mf_j\circ B_j(x)^{p_j}\frac{\det(dB_{m+1}(x)dB_{m+1}(x)^*)^{-\frac{1}{2}}d\sigma(x)}{\BL(\widetilde{\textsbf{dB}}(x),\widetilde{\textsbf{p}})}\\
         &=\lim_{\delta\rightarrow 0}\int_{A}\prod_{j=1}^mf_j\circ B_j(x)^{p_j}\frac{\chi_{\delta}\circ B_{m+1}(x)dx}{\BL(\widetilde{\textsbf{dB}}(x),\widetilde{\textsbf{p}})}.
    \end{align*}
    Applying Theorem \ref{algbras} inside the limit on the right-hand side we then obtain
    \begin{align*}
         \int_{A\cap M}\prod_{j=1}^mf_j\circ B_j(x)^{p_j}&\frac{dx}{\BL_{T_xM}(\textsbf{dB}(x),\textsbf{p})}\nonumber\\
         &\lesssim \deg(B_{m+1})\underset{\delta\rightarrow 0}{\lim}\left(\int_{\R^{n-d}}\chi_{\delta}(z)dz\right)\prod_{j=1}^m\left(\deg(B_j)\int_{ M_j}f_j(x_j)d\mu_j(x_j)\right)^{p_j}\\
         &\simeq\deg(M)\prod_{j=1}^m\left(\deg(B_j)\int_{ M_j}f_j(x_j)d\mu_j(x_j)\right)^{p_j},
    \end{align*}
    which yields the desired inequality, since the right-hand side is uniform in the choice of $A$, and extends to arbitrary $f_j\in L^1( M_j)$ via density.
\end{proof}

We shall henceforth assume that our domain is of full dimension, and to emphasise this, for the remainder of the proof we shall denote the domain of $B_j$ by $U\subset\R^n$ instead of $M$.

Having reduced Theorem \ref{genalgbras} to the euclidean case, we shall further reduce Theorem \ref{genalgbras} to a more discrete inequality, where the domain $U$ is replaced with a compact subset $\Omega\subset U$, and the arbitrary $L^1$ functions $f_j$ take the specific form of characteristic functions associated to small balls on $M_j$.

\begin{proposition}\label{discretealgbras}
    For every compact set $\Omega\subset U$, there exists a $\nu>0$ such that, for all $\delta\in(0,\nu)$ and all finite collections $\mathcal{V}_j$ (allowing duplicates) of $\delta$-balls in $ M_j$, the following holds:
\begin{align}
    \int_{\Omega}\prod_{j=1}^m\left(\sum_{V_j\in\mathcal{V}_j}\chi_{V_j}\circ B_j(x)\right)^{p_j}\frac{dx}{\BL(\textsbf{dB}(x),\textsbf{p})}\lesssim\prod_{j=1}^m\left(\deg(B_j)\delta^{n_j}\#\mathcal{V}_j\right)^{p_j}.\label{eq:discretealgbras}
\end{align}
\end{proposition}

\vspace{3pt}We may derive Theorem \ref{algbras} from Proposition \ref{discretealgbras} via a standard limiting argument, which we omit.

\subsection{Central Constructions}

The strategy for proving Proposition \ref{discretealgbras} is based on appealing to Theorem \ref{zorin}, in particular finding a collection of open subsets $H_1,...,H_m$ of algebraic varieties such that, if substituted into (\ref{eq:zorin}), then this inequality would yield (\ref{eq:discretealgbras}). These manifolds may be thought of as the unions of `discrete foliations' of the preimages $B_j^{-1}(V_j)$ via the fibres of $B_j$.\par
We shall now carry out this construction.
Fix $\Omega$ and let $\delta>0$ and $\mathcal{V}_j$ be a finite collection of $\delta$-balls in $M_j$. Let $\alpha>1$, for each $V_j\in\mathcal{V}_j$ let $x_{V_j}$ denote the centre of $V_j$, and choose an orthonormal basis $\partial_1,...,\partial_{n_j}\in T_{x_{V_j}} M_j$. Given $\epsilon>0$, we define the discrete $\epsilon$-grid $\Lambda^{\epsilon}_{V_j}:=\bigoplus_{i=1}^{n_j}\varepsilon\mathbb{Z}\partial_i$, and we consider the intersection of a dilation of $V_j$ with the image of this grid under the exponential map:
$$\Gamma(V_j):=\exp_{x_{V_j}}\left(\Lambda^{\delta^{\alpha}}_{V_j}\right)\cap 2V_j.$$
 \begin{figure}[h]
     \centering
     \includegraphics[width=6cm]{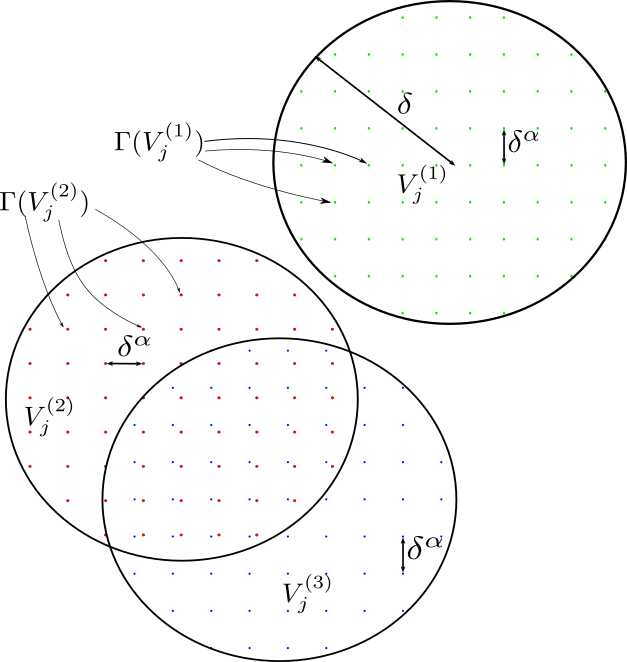}
     \caption{The specific case when $\mathcal{V}_j=\{V_j^{(1)},V_j^{(2)},V_j^{(3)}\}$}
     \label{fig: 1}
 \end{figure}
We have dilated the balls $V_j$ by a factor of $2$ for technical reasons that will become apparent in the proof of Lemma \ref{overlap}, the reader is encouraged to ignore it upon first reading. In order to track multiplicities, it shall be important that for each $V_j,V_j'\in\mathcal{V}_j$, we have $\Gamma(V_j)\cap\Gamma(V_j')=\emptyset$, however this is not guaranteed by our construction as it stands, hence if there exists $z\in\Gamma(V_j)\cap\Gamma(V_j')$, then we shall remedy this by simply translating one of these discrete sets by a negligible distance of, say, $\delta^{\alpha^{100}}$.\par
 
We shall now use the assumption that $B_j$ is quasialgebraic. For each $z\in M_j$ there exists a polynomial map $p_j^z:\R^n\rightarrow\R^{n_j}$ such that $B_j^{-1}(\{z\})$ is an open subset of $Z(p^z_j)$ and $\deg(p_j^z)\leq\deg(B_j)$. Define the following polynomial map:
 $$S_j:=\prod_{V_j\in\mathcal{V}_j}\prod_{z\in\Gamma(V_j)}p^z_j,$$
 and let $Z(S_j)$ be its zero-set. By our assumption that $B_j$ is a submersion, we may assume that $Z(S_j)$ is an $(n-n_j)$-dimensional variety, and contains the following open subset that will serve as our aforementioned `discrete' foliation:
 $$H_j:=\bigcup_{V_j\in\mathcal{V}_j}B_j^{-1}(\Gamma(V_j))\subset Z(S_j).$$
 
 Observe that if $\delta>0$ is chosen to be sufficiently small, then $\#\Gamma(V_j)\simeq \delta^{-\alpha n_j}|V_j|\simeq \delta^{(1-\alpha)n_j}$, hence we may bound the degree of $Z(S_j)$ as follows:
 \begin{align}
    \deg(Z(S_j))\leq\sum_{V_j\in\mathcal{V}_j}\sum_{z\in\Gamma(V_j)}\deg(p^z_j)\leq\deg(B_j)\sum_{V_j\in\mathcal{V}_j}\#\Gamma(V_j)\simeq \deg(B_j)\delta^{(1-\alpha)n_j}\#\mathcal{V}_j\label{eq:degree}
 \end{align}
 \begin{figure}[h]
     \centering
     \includegraphics[width=\linewidth]{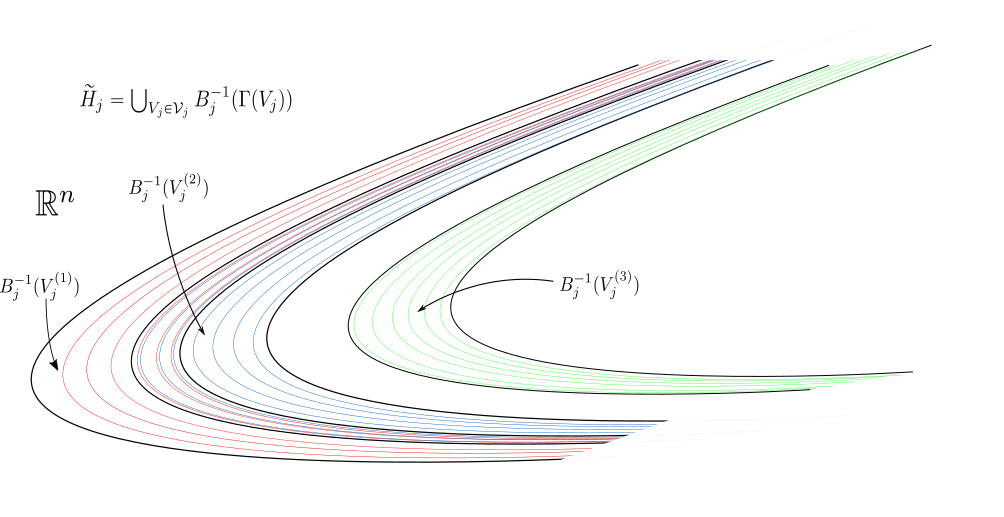}
     \caption{Picture of $H_j$}
     \label{fig: 2}
 \end{figure}
 \subsection{Heuristic Explanation of Proof Strategy}\label{heuristic}
 Let $f_j:=\sum_{V_j\in\mathcal{V}_j}\chi_{V_j}$, and observe that the right-hand side of (\ref{eq:degree}) is equal to $\deg(B_j)\delta^{-\alpha n_j}\int_{\mathbb{R}^{n_j}}f_j$, so provided we cancel the factor of $\delta^{\alpha n_j}$ at some stage, it then seems promising to substitute $H_1,...,H_m$ into (\ref{eq:zorin}), and try to obtain (\ref{eq:discretealgbras}) from there.\\

 Morally, we may view the left-hand side of (\ref{eq:zorin}) as measuring the size of the intersections of tubular neighbourhoods of the varieties $H_j$ of unit thickness, weighted by their mutual transversality. By rescaling we may reduce the size of these neighbourhoods to an arbitrarily small scale, for technical reasons we will reduce the thickness of the tubes to near $\delta^{\beta}$-scale, where $\alpha<\beta<1$.\\
 If we now substitute the varieties $H_1,...,H_m$ into (\ref{eq:zorin}), assuming our meshes $\Gamma(V_j)$ are sufficiently fine with respect to the size of $V_j$, then the left-hand side would essentially be measuring the size of the set
 \begin{align}
     \bigcap_{j=1}^m\bigcup_{V_j\in\mathcal{V}_j}\bigcup_{z\in\Gamma(V_j)}(B_j^{-1}(\{z\})+U_{\delta^{\beta}}(0)). \label{eq:heuristic}
 \end{align}
 which we claim contains $\bigcap_{j=1}^m\bigcup_{V_j\in\mathcal{V}_j}B_j^{-1}(V_j)\cap\Omega$, and it is this set that the left-hand side of (\ref{eq:discretealgbras}) is essentially measuring, so all we need to make sure of is that the two measures in question essentially coincide.\par
 The measure being applied to (\ref{eq:heuristic}) is the Lebesgue measure weighted not only by the transversality of the leaves $B_j^{-1}(\{z\})$ comprising $H_j$, as imparted by the integrand $BL'(\overrightarrow{T_{x_j}H_j},\textsbf{p})$, but also, for each $j$, by a combinatorial factor that counts, given $x\in\bigcap_{j=1}^m\bigcup_{V_j\in\mathcal{V}_j}B_j^{-1}(V_j)\cap\Omega$, the number of $\delta^{\beta}$--neighbourhoods that $x$ lies in, and this factor is given by $\sum_{z\in\Gamma(V_j)}\chi_{B_j^{-1}(\{z\})+U_{\delta^{\beta}}(0)}(x)$. As the forthcoming Lemma \ref{overlap} demonstrates, this factor itself splits into two factors: one counts the number of preimages $B_j^{-1}(V_j)$ that $x$ lies in, which is exactly given by $\sum_{V_j\in\mathcal{V}_j}\chi_{V_j}\circ B_j(x)$, and the other is a factor that counts the amount of overlap between tubes associated with the same ball $V_j$ at a point $x\in U$.
 \begin{figure}[h]
     \centering
     \includegraphics[width=\linewidth]{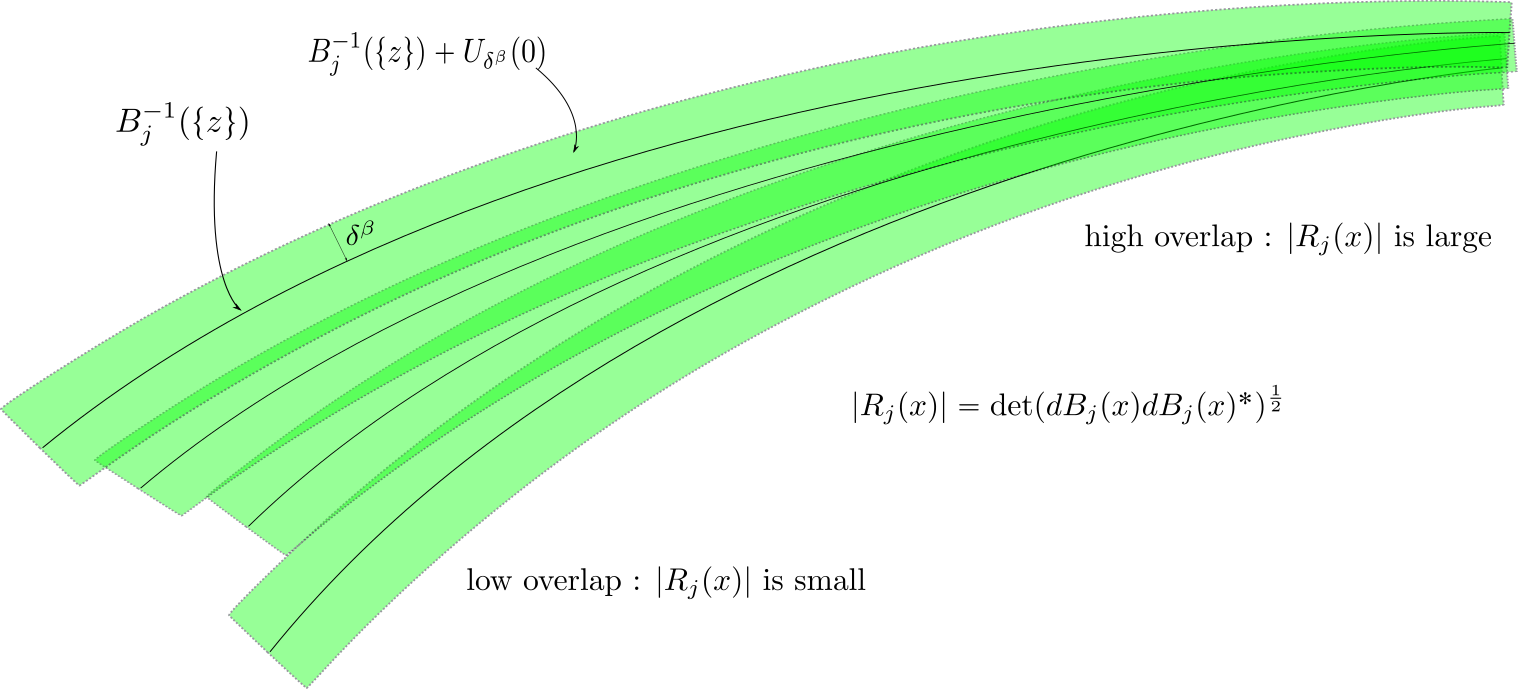}
     \caption{overlapping $\delta^{\beta}$-tubes}
     \label{fig: 3}
 \end{figure}
 This factor will be large when the tubes are tightly packed, and low when the tubes are more spaced out. These situations correspond to the derivative map $dB_j(x)$ having respectively large and small `volume', which is quantified by the function $|R_j(x)|$, which we define in the next section. It is due to the content of Lemma \ref{formula} that these additional $|R_j(x)|$-factors will allow us to move from $BL'$-factors to $\BL$-factors, which finally gives us the left-hand side of (\ref{eq:discretealgbras}).
 
 \section{Lemmas}

Here we shall prove the results that form the ingredients we need to prove Proposition \ref{discretealgbras}. First of all, we shall investigate how Fremlin tensor product norms behave under rescaling.
 
\begin{lemma}\label{rescale}
    Let $X_1,..,X_m\subset\R^n$ be smooth submanifolds such that $\dim(X_j)=k_j$, let $q_1,...,q_m\geq 1$, and let $F\in\bigotimes_{j=1}^m L^{q_j}(X_j)$. Then, for all $\epsilon>0$,
    \begin{align*}
        \Vert BL'(\overrightarrow{T_{x_j}X_j},\textsbf{p})\Vert_{\widebar{\bigotimes}_{j=1}^m L^{q_j}_{x_j}(X_j)}=\epsilon^{\sum k_j/q_j}\Vert BL'(\overrightarrow{T_{x_j}(\epsilon^{-1} X_j)},\textsbf{p})\Vert_{\widebar{\bigotimes}_{j=1}^mL^{q_j}_{x_j}(\varepsilon^{-1} X_j)}.
    \end{align*}
\end{lemma}
\begin{proof}
First of all, since dilation is a conformal mapping, it must preserve tangent spaces of submanifolds, so in particular $T_{\epsilon x_j}X_j=T_{x_j}(\epsilon^{-1} X_j)$. For each $j\in\{1,...,m\}$, let $F_j\in L^{q_j}(X_j)$ be an arbitrary function satisfying $F_j\geq 0$ and $BL'(\overrightarrow{T_{x_j}X_j},\textsbf{p})\leq F_1(x_1)...F_m(x_m)$ a.e. pointwise. By the definition of a Fremlin tensor product norm, it then suffices that 
\begin{align*}
     \prod_{j=1}^m\Vert F_j\Vert_{L^{q_j}(X_j)}=\epsilon^{\sum k_j/q_j}\prod_{j=1}^m\Vert F_j(\epsilon\cdot)\Vert_{L^{q_j}(\epsilon^{-1} X_j)},
\end{align*}
which follows immediately from rescaling the $L^{q_j}$ norms.
\end{proof}

 A necessary ingredient for proving Proposition \ref{discretealgbras} is a formula relating the standard $\BL$-constants with the nonstandard $BL'$-constants arising in (\ref{eq:zorin}). We find that we may derive an explicit factorisation that makes explicit the dual role that the BL-constant plays, in both measuring the mutual transversality of the kernels of the $L_j$ and measuring how close the maps $L_j$ come to being non-surjective.
\begin{lemma}\label{formula}
Let $(\textsbf{L},\textsbf{p})$ be a Brascamp--Lieb datum such that each map $L_j:V\rightarrow V_j$ is surjective, and let $R_j\in\Lambda^{n_j}(V)$ denote the $n_j$-fold wedge product of the rows of $L_j$ with respect to a suitable orthonormal basis, then
\begin{align}
\BL(\textsbf{L},\textsbf{p})=BL'(\overrightarrow{\ker(L_j)},\textsbf{p})\prod_{j=1}^m|R_j|^{-p_j}.
\end{align}\label{eq:formula}
\end{lemma}
\begin{proof}
   For the sakes of concreteness, we shall assume that the domains of the surjections $L_j$ is $\R^n$ equipped with the standard inner product. By the first isomorphism theorem, for each $j\in\{1,...,m\}$ there exists an isomorphism $\phi_j:\R^{n}/\ker(L_j)\rightarrow V_j$ such that $L_j=\phi_j\circ\pi_j$, where $\pi_j:\R^n\rightarrow\R^n/\ker(L_j)$ is the canonical projection map.\\
   First of all, we claim that $|\det(\phi_j)|=|R_j|$.
    To see this, observe that $|L_j[0,1]^n|=|\phi_j\circ\pi_j[0,1]^n|=|\det(\phi)|$, so the claim then follows provided we can show that $|L_j[0,1]^{n}|=|R_j|$.
\begin{align*}
    |L_j[0,1]^n|&=|(L_j[0,1]^n)\times[0,1]^{n-n_j}|= |M^{\top}[0,1]^n|=|\det(M)|,
\end{align*}
where $M\in\R^{n\times n}$ is the matrix whose first $n_j$ rows are the rows of $L_j$ and the last $n-n_j$ rows are $e_{n_j+1},...,e_n$, where $e_1,...,e_n$ is an orthonormal basis of $\mathbb{R}^n$ such that $e_1,...,e_{n_j}$ span $\ker(L_j)^{\bot}$ and $e_{n_j+1},...,e_n$ spans $\ker(L_j)$. Since $R_j=\pm|R_j|\bigwedge_{j=1}^{n_j}e_j$, the claim then quickly follows:
\begin{align*}
    |L_j[0,1]^n|= |\det M|&= |R_j\wedge(\bigwedge_{j=n_j+1}^{n}e_j)|=|R_j||\bigwedge_{j=1}^{n}e_j|=|R_j|.
\end{align*}
Now, let $f_j\in L^1(V_j)$ be arbitrary and $\widetilde{f}_j:=f_j\circ \phi_j$. We may then change variables and rewrite the left-hand side of the Brascamp--Lieb inequality associated to $(\textsbf{L},\textsbf{p})$ as follows.
\begin{align}
    \int_{\R^n}\prod_{j=1}^mf_j\circ L_j(x)^{p_j}dx =\int_{\R^n}\prod_{j=1}^m\widetilde{f}_j\circ\pi_j(x)^{p_j}dx =\int_{\R^n}\prod_{j=1}^m\widetilde{f}_j(x+\ker(L_j))^{p_j}dx \label{eq:formula1}
\end{align}
Moreover, $\int_{\R^n/\ker(L_j)}\widetilde{f}_j=|\det(\phi_j)|^{-1}\int_{ H_j}f_j=|R_j|^{-1}\int_{ H_j}f_j$, hence combining this with (\ref{eq:formula1}) we obtain that
\begin{align}
      \int_{\R^n}\prod_{j=1}^mf_j\circ L_j(x)^{p_j}dx \leq BL'(\overrightarrow{\ker(L_j)},\textsbf{p})\prod_{j=1}^m\left(|R_j|^{-1}\int_{ H_j}f_j\right)^{p_j}. \label{eq:formula2}
\end{align}
Therefore $\BL(\textsbf{L},\textsbf{p})\leq BL'(\overrightarrow{\ker(L_j)},\textsbf{p})\prod_{j=1}^m|R_j|^{-p_j}$. Furthermore, observing that \eqref{eq:formula2} is sharp, by the definitions of $\BL$ and $BL'$, this automatically improves to the desired formula $\BL(\textsbf{L},\textsbf{p})= BL'(\overrightarrow{\ker(L_j)},\textsbf{p})\prod_{j=1}^m|R_j|^{-p_j}$.
\end{proof}
We remark that $|R_j|$ may also be written as $\det(L_jL_j^*)^{1/2}$, since $|R_j|^2=\langle R_j,R_j\rangle_{\Lambda^{n_j}(\R^n)}=\det((r_{j,k}\cdot r_{j,l})_{k,l=1}^n)=\det(L_jL_j^*)$, where $r_{j,k}$ is the $k^{th}$ row of $L_j$. As one would expect, the formula \eqref{eq:formula1} also allows us to carry stability properties from the standard $\BL$-constants to the $BL'$-constants arising in \eqref{eq:zorin}, which we state more precisely in the following corollary.
\begin{Corollary}\label{BLlocconst}
Let $\Omega\subset U$ be compact. Writing $\textsbf{x}:=(x_1,...,x_m)$, the weight function $g:\Omega^m\rightarrow\R$  defined by $$g(\textsbf{x}):=BL'(\overrightarrow{\ker dB_j(x_j)},\textsbf{p})^{-1}$$ is uniformly continuous and locally constant at a sufficiently small scale, that is to say for $\epsilon>0$ sufficiently small depending on $\Omega$, for all $\textsbf{x},\textsbf{y}\in\Omega^m$,
$$|\textsbf{x}-\textsbf{y}|<\epsilon\Longrightarrow g(\textsbf{x})\lesssim g(\textsbf{y}).$$
\end{Corollary}
\begin{proof}
For each $x_j\in\Omega$, let $\Pi_j^{x_j}:\R^n\rightarrow \ker(dB_j(x_j))^{\bot}$ denote the projection map onto $\ker(dB_j(x_j))^{\bot}$, and let $\phi_j^{x_j}: \R^{n_j}\rightarrow \ker(dB_j(x_j))^{\bot}$ be a family of isometric isomorphisms that varies continuously in $x_j$. Define the family of surjections $L^{x_j}_j:\R^n\rightarrow \R^{n_j}$ by $L_j^{x_j}:=(\phi_j^{x_j})^{-1}\circ\Pi_j^{x_j}$, and let $\textsbf{L}^{\textsbf{x}}:=(L_j^{x_j})_{j=1}^m$. By Lemma \ref{formula}, $\BL(\textsbf{L}^{\textsbf{x}},\textsbf{p})^{-1}=g(\textsbf{x})^{-1}\prod_{j=1}^m|\det(\phi_j^{x_j})|^{p_j}=g(\textsbf{x})$, hence continuity of $g$ follows from the continuity of the reciprocal of the Brascamp--Lieb constant over $\Omega$, which was established in \cite{bennett2017behaviour}. By compactness of $\Omega$ and the positivity of $g$, $g\otimes g^{-1}$ is then uniformly continuous on $(\Omega^m)^2$, so because $g\otimes g^{-1}(\textsbf{x};\textsbf{x})=1$ for all $\textsbf{x}\in\Omega^m$, there exists $\epsilon>0$ such that for all $\textsbf{x},\textsbf{y}\in\Omega^m$, $g\otimes g^{-1}(\textsbf{x};\textsbf{y})<2$ provided that $|\textsbf{x}-\textsbf{y}|<\epsilon$, completing the proof.
\end{proof}

 The next proposition will allow us to simultaneously cover the preimages $B_j^{-1}(V_j)$ of the Balls $V_j$ by tubular neighbourhoods of the varieties comprising $H_j$, and account for the missing factor in the weight $\BL(\textsbf{dB}(x),\textsbf{p})^{-1}$, as alluded to in Section \ref{heuristic}.
\begin{lemma}\label{overlap}
Let $\Omega\subset U$ be compact and fix $j\in\{1,...,m\}$. Let $R_j(x)\in\Lambda^{n_j}(\R^n)$ denote the $n_j$-fold wedge product of the rows of $dB_j(x)$, then for a sufficiently small choice of $\delta>0$ depending on $\Omega$, over all $x\in \Omega$,
\begin{align}
   | R_j(x)|\chi_{V_j}\circ B_j(x)\lesssim\delta^{(\alpha-\beta)n_j}\sum_{z\in\Gamma(V_j)}\chi_{B_j^{-1}(\{z\})+U_{\delta^{\beta}}(0)}(x).\label{eq:overlap}
\end{align}
\end{lemma}
To prove this lemma, we shall need to establish the following intuitive geometric fact that shall allow us to deal with the nonlinearity present in the quasialgebraic maps $B_j$.
\begin{lemma}\label{linearise}
     Given the same hypotheses as Lemma \ref{overlap}, for a sufficiently small choice of $\delta>0$ depending on $\Omega$,  $L^x_j(U_{\delta/2}(x))\subset B_j(U_{\delta}(x))$ for all $x\in \Omega$, where $L_j^x(y):=\exp_{B_j(x)}(dB_j(x)(y-x))$ is now the first-order approximation of $B_j$ about $x$ (not to be confused with the notation used in Corollary \ref{BLlocconst}).
\end{lemma}
\begin{proof}
 Fix some $x\in\Omega$. Let $T:=\exp_{B_j(x)}\circ(dB_j(x)dB_j(x)^*)^{-1/2}\circ \exp_{B_j(x)}^{-1}$, then for $\delta>0$ smaller than the minimal injectivity radius among $z\in B_j(\Omega)$.
    \begin{align*}
        L^x_j(U_{\delta/2}(x))\subset B_j(U_{\delta}(x))\iff TL^x_j(U_{\delta/2}(x))\subset TB_j(U_{\delta}(x))
    \end{align*}
    Hence we may assume without loss of generality that $dB_j(x)$ is a projection, in the sense that $dB_j(x)dB_j(x)^*=I_{T_{B_j(x)}M_j}$, and thus we may also assume that $L_j(U_{\delta/2}(x))=U_{\delta/2}(B_j(x))$.
    It then suffices to show that for $\delta>0$ sufficiently small, $\partial B_j(U_\delta(x))\cap U_{\delta/2}(B_j(x))=\emptyset$, in other words that for all $z\in\partial B_j(U_{\delta}(x))$, $d(z,B_j(x))>\delta/2$. First of all, $\partial B_j(U_{\delta}(x))=B(\partial U_{\delta}(x))$, so for a given $z\in\partial B_j(U_{\delta}(x))$, there exists a $y\in \partial U_{\delta}(x)$ such that $\exp_{B_j(x)}\circ dB_j(x)(y)=z$. By Taylor's theorem, there exists a $c>0$ depending on $\Omega$ such that, for $\delta>0$ sufficiently small,
    \begin{align*}
        d(z,B_j(x))&=|dB_j(x)(y-x)|+\mathcal{O}(|y-x|^2)\\
                &\geq\delta-c\delta^2>\delta/2
    \end{align*}
\end{proof}
\begin{proof}[Proof of Lemma \ref{overlap}]
 We immediately have that for each $x\in \Omega$,
\begin{align*}
 \sum_{z\in\Gamma(V_j)}\chi_{B_j^{-1}(\{z\})+U_{\delta^{\beta}}(0)}(x)&=\#\{z\in\Gamma(V_j):d(x,B_j^{-1}(\{z\}))\leq\delta^{\beta}\}\\
 &=\#\left(\Gamma(V_j)\cap B_j(U_{\delta^{\beta}}(x))\right)\\
 &=\#\left(\exp_{x_{V_j}}\left(\Lambda^{\delta^{\alpha}}_{V_j}\right)\cap 2V_j\cap B_j(U_{\delta^{\beta}}(x))\right).
 \end{align*}
 By Lemma \ref{linearise} we then, for $\delta>0$ sufficiently small, have the bound
 \begin{align}
  \sum_{z\in\Gamma(V_j)}\chi_{B_j^{-1}(\{z\})+U_{\delta^{\beta}}(0)}(x)\geq
 \#(\exp_{x_{V_j}}\left(\Lambda^{\delta^{\alpha}}_{V_j}\right)\cap 2V_j\cap L^x_j(U_{\delta^{\beta}/2}(x))).\label{eq:overlap1}
 \end{align}
 %Now, $\partial (B_j(U_{\delta}(x)))$ partitions $\mathbb{R}^{n_j}$ into $B_j(U_{\delta}(x))^{\circ}$ and $(\mathbb{R}^n\setminus B_j(U_{\delta}(x)))^{\circ}$, so $L_j^x(U_{\delta/2}(x))$ is contained in one of these sets because (\ref{eq:separation}) implies that $L_j^x(U_{\delta/2}(x))\cap\partial B_j(U_{\delta}(x))=\emptyset$, but $B_j(x)\in L_j^x(U_{\delta/2}(x))\cap B_j(U_{\delta}(x))^{\circ}$, hence we must conclude that $ L_j^x(U_{\delta/2}(x))\subset B_j(U_{\delta}(x))^{\circ}$, which proves the claim, and therefore (\ref{eq:overlap1}) holds.

Recall that we denote the centre of $V_j$ by $x_{V_j}\in M_j$. $\vert dB_j(x)\vert$ is uniformly bounded over $x\in\Omega$, so provided that $x\in B_j^{-1}(V_j)$, then for all $y\in U_{\delta^{\beta}/2}(x)$, $d(L^x_j(y),x_{V_j})\leq d(B_j(x),x_{V_j})+\Vert dB_j(x)\Vert_{L^{\infty}(\Omega)} |y-x|\leq \delta+\Vert dB\Vert\delta^{\beta}/2< 2\delta$, if we take $\delta>0$ to be sufficiently small. This implies that if $x\in B_j^{-1}(V_j)\cap \Omega$, then for $\delta>0$ sufficiently small, $L^x_j(U_{\delta^{\beta}/2}(x))\subset 2V_j$, which together with \eqref{eq:overlap1} yields that
 \begin{align}
 \sum_{z\in\Gamma(V_j)}\chi_{B_j^{-1}(\{z\})+U_{\delta^{\beta}}(0)}(x)&\geq\#\left(\exp_{x_{V_j}}\left(\Lambda^{\delta^{\alpha}}_{V_j}\right)\cap L^x_j(U_{\delta^{\beta}/2}(x))\right)\chi_{B_j^{-1}(V_j)}(x)\nonumber\\
 &\simeq\#\left(\exp_{x_{V_j}}\left(\Lambda^{\delta^{\alpha-\beta}}_{V_j}\right)\cap L^x_j(U_1(x))\right)\chi_{B_j^{-1}(V_j)}(x).\label{eq:overlap2}
 \end{align}
 Given $\varepsilon>0$, define $\mathcal{Q}_j^{\varepsilon}$ to be the cubic decomposition of $T_{x_{V_j}}M_j$ into $\varepsilon$-cubes whose sides are axis parallel and whose corresponding set of centres is $\Lambda^{\epsilon}_{V_j}$, and recall the definition of $c>0$ from the proof of Lemma \ref{linearise}. If we take $\delta^{\alpha-\beta}<c/10$, then for all $x\in \Omega$ and $Q\in\mathcal{Q}_j^{\delta^{\alpha-\beta}}$ such that $Q\cap L^x_j(U_{1/2}(x))\neq\emptyset$, we must have that $Q\subset L^x_j(U_{1}(x))$, since otherwise there would exist a point outside of $L^x_j(U_1(x))$ within a distance $c/2$ of $B_j(x)$, which implies that $dB_j(x)\vert_{\ker dB_j(x)^{\bot}}$ has an eigenvalue with absolute value less than $c$, which is of course a contradiction. Since the map that takes a cube in $\mathcal{Q}_j^{\delta^{\alpha-\beta}}$ to its centre then defines an injection from $D:=\{Q\in\mathcal{Q}_j^{\delta^{\alpha-\beta}}:Q\cap L^x_j(U_{1/2}(x))\neq\emptyset\}$ to $\exp_{x_{V_j}}\left(\Lambda^{\delta^{\alpha-\beta}}_{V_j}\right)\cap L^x_j(U_1(x))$, we obtain the following bound:
 \begin{align*}
 \sum_{z\in\Gamma(V_j)}\chi_{B_j^{-1}(\{z\})+U_{\delta^{\beta}}(0)}(x)&\geq(\#D)\chi_{B_j^{-1}(V_j)}(x)\\
 &=\big\vert\bigcup_{Q\in D}Q\big\vert\big\vert[0,\delta^{\alpha-\beta}]^{n_j}\big\vert^{-1}\chi_{B_j^{-1}(V_j)}(x)\\
 &\geq|L^x_j(U_{1/2}(x))|\delta^{(\beta-\alpha)n_j}\chi_{B_j^{-1}(V_j)}(x)\\
 &\simeq|dB_j(x)[0,1]^n|\delta^{(\beta-\alpha)n_j}\chi_{B_j^{-1}(V_j)}(x).
\end{align*}
Since $\chi_{B_j^{-1}(V_j)}=\chi_{V_j}\circ B_j$, the claim then follows from the fact that $|dB_j(x)[0,1]^{n}|=|R_j(x)|$, which follows from an inspection of the proof that $|L_j[0,1]^n|=|R_j|$ in Lemma \ref{formula}.
\end{proof}
Finally, we need a technical lemma that will allow us to bound the volumes of intersections of balls with varieties below by the characteristic functions arising on the right-hand side of \eqref{eq:overlap}.
\begin{lemma}\label{volumebound}
 Let $\Omega\subset U$ be compact, and fix $j\in\{1,...,m\}$. Then, for a sufficiently small choice of $\delta>0$ depending on $\Omega$, the following holds for all $x\in\Omega$ and $z\in M_j$:
 \begin{align}
     \delta^{\beta(n-n_j)}\chi_{B_j^{-1}(\{z\})+U_{\delta^{\beta}}(0)}(x)\lesssim|B_j^{-1}(\{z\})\cap U_{2\delta^{\beta}}(x)|.\label{eq:volumebound}
 \end{align}
\end{lemma}
\begin{proof}
We shall begin with some reductions. First of all, we fix $z\in M_j$, making sure in what comes after that our choice $\delta>0$ does not depend on this particular choice of $z\in M_j$. Suppose that for each choice of $x_0\in\Omega$, there exists a corresponding choice of $\delta_{x_0}>0$ such that \eqref{eq:volumebound} holds for each $x\in U_{\delta_{x_0}^2}(x_0)$ and $0<\delta\leq\delta_{x_0}$. The set $\{U_{\delta_{x_0}^2}(x_0):x_0\in\Omega\}$ is then an open cover of $\Omega$, so by compactness of $\Omega$ we may take a finite subcover $\mathcal{U}$. The minimal radius among the balls in $\mathcal{U}$, which we shall denote by $\tilde{\delta}$, is such that \eqref{eq:volumebound} holds for all $\delta\in(0,\tilde{\delta})$ and $x\in \Omega$, so the lemma would then hold. It therefore suffices to fix $x_0\in\Omega$ and prove the claim that there exists a $\delta_{x_0}$ such that \eqref{eq:volumebound} holds for each $x\in U_{\delta_{x_0}^2}(x_0)$ and $0<\delta\leq\delta_{x_0}$.\par 
Furthermore, we may assume that $M_j$ is an open subset of $\R^{n_j}$. To justify this, by compactness of $\Omega$ and continuity of $B_j$, we may choose a $\delta>0$ sufficiently small such that $\exp_x$ is a diffeomorphism on $U_{\delta}(0)\subset T_yM_j$ for each $y\in B_j(\Omega)$. We then restrict $B_j$ to $B_j^{-1}(U_{\delta}(z))$ and prove that the claim holds with $B_j$ replaced with $\widetilde{B}_j:=\exp^{-1}_z\circ B_j$, and $z$ replaced with $0\in\R^{n_j}$, since in this case  $\widetilde{B}_j^{-1}(\{0\})=B_j^{-1}(\{z\})$, hence we would obtain the claim for our original choice of $B_j$.\par
Fix $x_0\in \Omega$, recall the definition of $L^{x_0}_j$ from Lemma \ref{overlap} and let $A\in SO(n)$ be a rotation such that $A\ker dB_j(x_0)=\mathbb{R}^{n-n_j}\times\{0\}^{n_j}$. Since $B_j$ is a submersion on $\Omega$, $dB_j(x_0)$ is surjective, hence it admits a right inverse, call it $S$. Let $\psi:=B_j-dB_j(x_0)$. We define the function $\phi:\mathbb{R}^n\rightarrow\mathbb{R}^n$ by
$$\phi(y):=A(y+S\psi(y))).$$
For all $y\in B_j^{-1}(\{z\})$, $z=B_j(y)=dB_j(x_0)y+\psi(y)=dB_j(x_0)(y+S\psi(y))=dB_j(x_0)(A^{-1}\phi(y))$, so $A^{-1}\phi(y)\in dB_j(x_0)^{-1}(\{z\})$, hence $A^{-1}\phi(y)-Sz\in\ker dB_j(x_0)$, so $\phi(y)\in\mathbb{R}^{n-n_j}\times\{0\}^{n_j}+ASz$. We have now shown that $\phi(B_j^{-1}(\{z\}))\subset\mathbb{R}^{n-n_j}\times\{0\}^{n_j}+ASz$. Moreover, one quickly verifies that $d\phi(x_0)=A(I+Sd\psi(x_0))=A$, hence $\phi$ is a diffeomorphism in a sufficiently small ball around $x_0$, therefore by taking $\delta$ to be sufficiently small, we may assume that, for all $x\in U_{\delta^{2\beta}}(x_0)$,  $U_{\frac{3\delta^{\beta}}{2}}(\phi(x))\subset \phi(U_{2\delta^{\beta}}(x))$ and  $\det(d\phi|_{B_j^{-1}(\{z\})}(y))\simeq 1$ for all $y\in U_{\delta}(x)$, from which it follows that, for all $x\in U_{\delta^{2\beta}}(x_0)$,
\begin{align}
    |B_j^{-1}(\{z\})\cap U_{2\delta^{\beta}}(x)|&=\int_{(\mathbb{R}^{n-n_j}\times\{0\}+ASz)\cap\phi(U_{2\delta^{\beta}}(x))}\det(d\phi|_{B_j^{-1}(\{z\})}(y))^{-1}dy\nonumber\\
    &\simeq|(\mathbb{R}^{n-n_j}\times\{0\}^{n_j}+ASz)\cap\phi(U_{2\delta^{\beta}}(x))|\nonumber\\
    &\geq|(\mathbb{R}^{n-n_j}\times\{0\}^{n_j}+ASz)\cap U_{\frac{3\delta^{\beta}}{2}}(\phi(x))|\nonumber\\
    &\geq|(\mathbb{R}^{n-n_j}\times\{0\}^{n_j}+ASz)\cap U_{\frac{3\delta^{\beta}}{2}}(\phi(x))|\chi_{B_j^{-1}(\{z\})+ U_{\delta^{\beta}}(0)}(x). \label{eq:volumebound1}
\end{align}
Since $\phi$ is smooth and $d\phi(x_0)=A$ is an isometry, if $x_0\in B_j^{-1}(\{z\})+U_{\delta^{\beta}}(0)$ and $\delta$ is sufficiently small then by Taylor's theorem we know that for all $x\in U_{\delta^{2\beta}}(x_0)$, $\phi(x)\in \phi(B_j^{-1}(\{z\}))+U_{\frac{5\delta^{\beta}}{4}}(0)=(\mathbb{R}^{n-n_j}\times\{0\}^{n_j}+ASz)+U_{\frac{5\delta^{\beta}}{4}}(0)$. In other words, $\dist(\phi(x),(\mathbb{R}^{n-n_j}\times\{0\}^{n_j}+ASz))\leq\frac{5\delta^{\beta}}{4}$, hence $(\mathbb{R}^{n-n_j}\times\{0\}^{n_j}+ASz)\cap U_{\frac{3\delta^{\beta}}{2}}(\phi(x)))$ is an $(n-n_j)$-disc of radius at least $\sqrt{\frac{9\delta^{2\beta}}{4}-\frac{25\delta^{2\beta}}{16}}\simeq\delta^{\beta}$, therefore
\begin{align}
    |(\mathbb{R}^{n-n_j}\times\{0\}^{n_j}+ASz)\cap U_{\frac{3\delta^{\beta}}{2}}(\phi(x)))|\chi_{B_j^{-1}(\{z\})+ U_{\delta^{\beta}}(0)}(x)\gtrsim\delta^{\beta(n-n_j)}\chi_{B_j^{-1}(\{z\})+ U_{\delta^{\beta}}(0)}.\label{eq:volumebound2}
\end{align}
This bound together with (\ref{eq:volumebound1}) then yields the claim.
\end{proof}
\section{Proof of Proposition \ref{discretealgbras}}
Let $\Omega\subset U$ and choose $\delta>0$ so that we may apply Corollary \ref{BLlocconst}, Lemma \ref{overlap}, and Lemma  \ref{volumebound} to $\Omega$. After first applying Lemma \ref{formula}, they yield the following pointwise estimate for all $x\in\Omega$,
\begin{align}
    \BL(\textsbf{dB}(x)&,\textsbf{p})^{-1}\prod_{j=1}^m\left(\sum_{V_j\in\mathcal{V}_j}\chi_{V_j}\circ B_j(x)\right)^{p_j}=BL'(\overrightarrow{\ker dB_j(x)},\textsbf{p})^{-1}\prod_{j=1}^m\left(\sum_{V_j\in\mathcal{V}_j}|R_j(x)|\chi_{V_j}\circ B_j(x)\right)^{p_j}\nonumber\\
    &\lesssim BL'(\overrightarrow{\ker dB_j(x)},\textsbf{p})^{-1}\prod_{j=1}^m\delta^{(\alpha-\beta)p_jn_j}\left(\sum_{V_j\in\mathcal{V}_j}\sum_{z\in\Gamma(V_j)}\chi_{B_j^{-1}(\{z\})+U_{\delta^{\beta}}(0)}(x)\right)^{p_j}\nonumber\\
    &=\delta^{(\alpha-\beta)n}BL'(\overrightarrow{\ker dB_j(x)},\textsbf{p})^{-1}\prod_{j=1}^m\left(\sum_{V_j\in\mathcal{V}_j}\sum_{z\in\Gamma(V_j)}\chi_{B_j^{-1}(\{z\})+U_{\delta^{\beta}}(0)}(x)\right)^{p_j}\nonumber\\
    &\lesssim \delta^{(\alpha-\beta)n}BL'(\overrightarrow{\ker dB_j(x)},\textsbf{p})^{-1}\prod_{j=1}^m\delta^{-\beta p_j(n-n_j)}\left(\sum_{V_j\in\mathcal{V}_j}\sum_{z\in\Gamma(V_j)}|B^{-1}_j(\{z\})\cap U_{2\delta^{\beta}}(x)|\right)^{p_j}\nonumber\\
    &=\delta^{(\alpha-\beta P)n}BL'(\overrightarrow{\ker dB_j(x)},\textsbf{p})^{-1}\prod_{j=1}^m|H_j\cap U_{2\delta^{\beta}}(x)|^{p_j}.\label{eq:proof1}
\end{align}
Above we used the scaling condition $\sum_{j=1}p_jn_j=n$ to pull out the power of $\delta$ from the product. By Corollary \ref{BLlocconst}, for all $x\in\Omega$ and $x_1,...,x_m\in H_j\cap U_{2\delta^{\beta}}(x)$,
\begin{align}
    BL'(\overrightarrow{\ker dB_j(x)},\textsbf{p})^{-\frac{1}{P}}\simeq BL'(\overrightarrow{T_{x_j}H_j},\textsbf{p})^{-\frac{1}{P}}. \label{BLlocconst2}
\end{align}
We may then average \eqref{BLlocconst2} via the Fremlin tensor product norm to find that find that
\begin{align}
    BL'(\overrightarrow{\ker dB_j(x)},\textsbf{p})^{-1}\prod_{j=1}^m|H_j\cap U_{2\delta^{\beta}}(x)|^{p_j}\simeq \Vert BL'(\overrightarrow{T_{x_j}H_j},\textsbf{p})^{\frac{-1}{P}}\Vert_{\widebar{\bigotimes}_{j=1}^m L^{P/p_j}_{x_j}(H_j\cap U_{2\delta^{\beta}}(x))}^P.\label{eq:BLlocconst}
\end{align}
We then integrate the inequality \eqref{eq:proof1} combined with \eqref{eq:BLlocconst} with respect to $x$ over $\Omega$.
\begin{align*}
    \int_{\Omega}\prod_{j=1}^m\left(\sum_{V_j\in\mathcal{V}_j}\chi_{V_j}\circ B_j(x)\right)^{p_j}&\frac{dx}{\BL(\textsbf{dB}(x),\textsbf{p})}\nonumber\\
    &\lesssim \delta^{(\alpha-\beta P)n} \int_{\Omega}\Vert BL'(\overrightarrow{T_{x_j}H_j},\textsbf{p})^{\frac{-1}{P}}\Vert_{\widebar{\bigotimes}_{j=1}^mL^{P/p_j}_{x_j}(H_j\cap U_{2\delta^{\beta}}(x))}^P
\end{align*}
At this point we then apply Lemma \ref{rescale} to rescale the inner integral so that we may then apply Theorem \ref{zorin}. Finally, using the bound on the degree of $Z(S_j)\supset H_j$ \ref{eq:degree}, we obtain (\ref{eq:discretealgbras}), completing the proof.
\begin{align*}
 \int_{\Omega}&\prod_{j=1}^m\left(\sum_{V_j\in\mathcal{V}_j}\chi_{V_j}\circ B_j(x)\right)^{p_j}\frac{dx}{\BL(\textsbf{dB}(x),\textsbf{p})}\\
    &\lesssim \delta^{\beta\sum_{j=1}^mp_j(n-n_j)}\delta^{(\alpha-\beta(P-1))n} \int_{\tfrac{\delta^{-\beta}}{2}\Omega}\Vert BL'(\overrightarrow{T_{x_j}\left(\tfrac{\delta^{-\beta}}{2}H_j\right)},\textsbf{p})^{\frac{-1}{P}}\Vert_{\widebar{\bigotimes}_{j=1}^mL^{P/p_j}_{x_j}\left(\tfrac{\delta^{-\beta}}{2}H_j\cap U_1(x)\right)}^Pdx\\
    &\leq \delta^{\alpha n} \int_{\R^n}\Vert BL'(\overrightarrow{T_{x_j}\left(\tfrac{\delta^{-\beta}}{2}H_j\right)},\textsbf{p})^{\frac{-1}{P}}\Vert_{\widebar{\bigotimes}_{j=1}^mL^{P/p_j}_{x_j}\left(\tfrac{\delta^{-\beta}}{2}H_j\cap U_1(x)\right)}^Pdx\\
    &\leq \delta^{\alpha n} \int_{\R^n}\Vert BL'(\overrightarrow{T_{x_j}\left(\tfrac{\delta^{-\beta}}{2}Z(S_j)\right)},\textsbf{p})^{\frac{-1}{P}}\Vert_{\widebar{\bigotimes}_{j=1}^mL^{P/p_j}_{x_j}\left(\tfrac{\delta^{-\beta}}{2}Z(S_j)\cap U_1(x)\right)}^Pdx\\
    &\lesssim \delta^{\alpha n}\prod_{j=1}^m\left(\deg Z(S_j)\right)^{p_j}
    \lesssim \delta^{\alpha n}\prod_{j=1}^m\left(\deg(B_j)\delta^{(1-\alpha)n_j}\#\mathcal{V}_j\right)^{p_j}=\prod_{j=1}^m\left(\deg(B_j)\delta^{n_j}\#\mathcal{V}_j\right)^{p_j}
\end{align*}
\qed
\section{Appendix}\label{appendix}
Here we prove some of the technical lemmas stated in Sections \ref{intro} and \ref{setup} and the application to Young's inequality on algebraic groups (Corollary \ref{youngs}).
\begin{proof}[Proof of Lemma \ref{sigmapprox}]
Let $w(x):=\det(dB(x)dB(x)^*)^{-\frac{1}{2}}$. The lemma follows from the co-area formula and the continuity of the quantity $\int_{B^{-1}(\{z\})}f(x)w(x)d\sigma(x)$ in $z\in N$, since we then have that
    \begin{align*}
        &\left\vert\int_{A}f(x)\chi_{\delta}\circ B(x)dx-\int_{B^{-1}(\{z_0\})}f(x)w(x)d\sigma(x)\right\vert\\
        &=\delta^{d-n}\left\vert\int_{U_{\delta}(0)}\left(\int_{A\cap B^{-1}(\{z\})}f(x)w(x)d\sigma(x)-\int_{A\cap B^{-1}(\{z_0\})}f(x)w(x)d\sigma(x)\right)dz\right\vert\nonumber\\
        &\lesssim\left\Vert \int_{A\cap B^{-1}(\{z\})}f(x)w(x)d\sigma(x)-\int_{A\cap B^{-1}(\{z_0\})}f(x)w(x)d\sigma(x) \right\Vert_{L^{\infty}_z(U_{\delta}(z_0))}\\
        &\underset{\delta\rightarrow 0}{\longrightarrow} 0.
    \end{align*}
\end{proof}
\begin{proof}[Proof of Lemma \ref{BLconstants}]
The fact that the scaling condition is satisfied is trivial. As for the second claim, we shall first prove that $\BL(\textsbf{L},\textsbf{p})\leq\det(L_{m+1}L_{m+1}^*)^{\frac{1}{2}}\BL(\widetilde{\textsbf{L}},\widetilde{\textsbf{p}})$. If we let $\chi_{\delta}:\R^{n-d}\rightarrow\R$ be as in Lemma \ref{sigmapprox}, with $z_0=0$, and we take arbitrary $f_j\in L^1(\R^{n_j})$, then by Lemma \ref{sigmapprox}, we have that
\begin{align*}
    \int_V\prod_{j=1}^mf_j\circ L_j(x)^{p_j}dx&=\det(L_{m+1}L_{m+1}^*)^{\frac{1}{2}}\underset{\delta\rightarrow 0}{\lim}\int_{\R^n}\left(\prod_{j=1}^mf_j\circ L_j(x)^{p_j}\right)\chi_{\delta}\circ L_{m+1}(x)dx\\
    &\leq\det(L_{m+1}L_{m+1}^*)^{\frac{1}{2}}\BL(\widetilde{\textsbf{L}},\widetilde{\textsbf{p}})\prod_{j=1}^m\left(\int_{ M_j}f_j\right)^{p_j},
\end{align*}
which establishes that $\BL(\textsbf{L},\textsbf{p})\leq\det(L_{m+1}L_{m+1}^*)^{\frac{1}{2}}\BL(\widetilde{\textsbf{L}},\widetilde{\textsbf{p}})$. We shall now prove the converse inequality. For each $1\leq j\leq m+1$, let $f_j\in C^\infty(\R^{n_j})$ be a smooth function with unit mass. The claim quickly follows upon decomposing $\R^n$ into $V\oplus V^{\bot}$ and applying the Brascamp-Lieb inequality associated to the datum $(\textsbf{L},\textsbf{p})$ to the integral over $V$:
\begin{align*}
    \int_{\R^n}\prod_{j=1}^{m+1}f_j\circ L_j(x)^{p_j}dx&=\int_{V^{\bot}}\left(\int_{V}\prod_{j=1}^mf_j(L_j(x)+ L_j(y))^{p_j}dx\right)f_{m+1}\circ L_{m+1}(y)dy\\
    &\leq \BL(\textsbf{L},\textsbf{p})\int_{V^{\bot}}f_{m+1}\circ L_{m+1}(y)\prod_{j=1}^m\left(\int_{ M_j}f_j(z+L_j(y))dz\right)^{p_j}dy\\
    &=\BL(\textsbf{L},\textsbf{p})\det(L_{m+1}L_{m+1}^*)^{-\frac{1}{2}}.
\end{align*}
\end{proof}
\begin{proof}[Proof of Corollary \ref{youngs}]
 By duality, (\ref{eq:youngs}) is equivalent to the bound
\begin{align}
    \int_{G}\phi(x)\left(\Asterisk_{j=1}^mf_j\Delta^{\sum_{l=1}^{j-1}\frac{1}{p_l'}}\right)(x)d\mu(x)\lesssim\deg(G)\Vert\phi\Vert_{L^{r'}(G)}\prod_{j=1}^m\Vert f_j\Vert_{L^{p_j}(G)}. \label{eq: youngsdual}
\end{align}
For $1\leq j\leq m$, define the nonlinear maps $B_j:G^m\rightarrow G$, $B_j(x_1,...,x_m):=x_j$, and $B_{m+1}:G^m\rightarrow G$, $B_{m+1}(x_1,...,x_m):=\prod_{j=1}^m x_j$. Deleting the null set of singular points from their ranges, these maps are quasialgebraic of degree $1$ for $1\leq j\leq m$, and $\deg(B_{m+1})\leq \deg(m_G)$, hence by Theorem \ref{genalgbras} we know that
\begin{align}
     \int_{G^m}\phi\left(\prod_{j=1}^mx_j\right)\prod_{j=1}^mf_j(x_j)\frac{d\sigma_1(x_1)...d\sigma_m(x_m)}{\BL_{T_{\underline{x}}G^m}(\textsbf{dB}(\underline{x}),\textsbf{p})}\lesssim\deg(G)\deg(m_g)^{\sigma}\Vert\phi\Vert_{L^{r'}(G)}\prod_{j=1}^m\Vert f_j\Vert_{L^{p_j}(G)} 
\end{align}
where $\underline{x}:=(x_1,...,x_m)$, which is equivalent to (\ref{eq: youngsdual}) provided we have the identity
\begin{align}
    \BL_{T_{\underline{x}}G^m}(\textsbf{dB}(\underline{x}),\textsbf{p})=B_{\textsbf{p},n}\prod_{j=1}^m\omega(x_j)^{-1}\Delta(x_j)^{-\sum_{l=1}^{j-1}\frac{1}{p_j'}} \label{eq: youngsid}
\end{align}
at all configurations of smooth points $x_1,...,x_m$ of $G$, where $d\mu(x)=\omega(x) d\sigma(x)$, and $B_{\textsbf{p},n}$ is the best constant for the $n$-dimensional euclidean multilinear Young's inequality associated to the exponents $\textsbf{p}:=(p_1,...,p_m)$. Let $x_1,...,x_m\in G$ be smooth points, the left-hand side of (\ref{eq: youngsid}) is by definition the best constant $C>0$ in the inequality
\begin{align}
    \int_{\prod_{j=1}^mT_{x_j}G}\phi\left(\sum_{j=1}^mx_1...x_{j-1}v_jx_{j+1}...x_m\right)\prod_{j=1}^mf_j(v_j)dv_j\leq C\Vert\phi\Vert_{L^{r'}(T_{x_1...x_m}G)}\prod_{j=1}^m\Vert f_j \Vert_{L^{p_j}(T_{x_j}G)}, \label{eq: youngsder}
\end{align}
where the Lebesgue measure on the left-hand side is induced by the Lebesgue measure on the ambient euclidean space, and the Lebesgue measures defining the norms on the right-hand side are induced by the left-invariant Riemannian metric on $G$.\par
First of all, we multiply the measure on the left by the constant $\prod_{j=1}^m\omega(x)$ for convenience. We then apply the linear transformation from the Lie algebra $\frak{g}$ to $T_{x_j}G$ defined by the mapping $v_j\mapsto x_1...x_m(x_1...x_{j-1})^{-1}v_j(x_{j+1}...x_m)^{-1}$, this is to turn the left-hand side of (\ref{eq: youngsder}) into an integral to which we may directly apply the euclidean Young's inequality:
\begin{align*}
    &\int_{\prod_{j=1}^mT_{x_j}G}\phi\left(\sum_{j=1}^mx_1...x_{j-1}v_jx_{j+1}...x_m\right)\prod_{j=1}^mf_j(x_j)\omega(x_j)dv_j\nonumber\\
    &=\int_{\frak{g}^m}\phi\left(x_1...x_m\sum_{j=1}^mv_j\right)\prod_{j=1}^mf_j((x_1...x_{j-1})^{-1}x_1...x_mv_j(x_{j+1}...x_m)^{-1})dx_j\Delta(x_{j+1}...x_m)^{-1}dv_j\\
    &\leq B_{\textsbf{p},n}\Vert \phi(x_1...x_mv)\Vert_{L^{r'}_v(\frak{g})}\prod_{j=1}^m\Delta(x_{j+1}...x_m)^{-1}\Vert f_j((x_{j+1}...x_m)^{-1}x_1...x_mv_j(x_{j+1}...x_m)^{-1})\Vert_{L^{p_j}_{v_j}(\frak{g})}\\
    &=B_{\textsbf{p},n}\Vert \phi\Vert_{L^{r'}(T_{x_1...x_m}G)}\prod_{j=1}^m\Delta(x_{j+1}...x_m)^{\frac{1}{p_j}-1}\Vert f_j\Vert_{L^{p_j}(T_{x_j}G)}\\
    &=B_{\textsbf{p},n}\Vert \phi\Vert_{L^{r'}(T_{x_1...x_m}G)}\prod_{j=1}^m\Delta(x_j)^{-\sum_{l=1}^{j-1}\frac{1}{p_l'}}\Vert f_j\Vert_{L^{p_j}(T_{x_j}G)}.
\end{align*}
Since this inequality is sharp by definition of $B_{\textsbf{p},n}$, we have established (\ref{eq: youngsid}), thus completing the proof.
\end{proof}
\bibliographystyle{plain}
\bibliography{bib}

\end{document}